\documentclass[review, nopreprintline]{elsarticle}
\usepackage[a4paper, left=2cm, right=2cm, top=4cm, bottom=4cm]{geometry}

\usepackage[english]{babel}
\usepackage[utf8]{inputenc}
\usepackage[T1]{fontenc}
\usepackage{color}
\usepackage{lmodern, microtype}

\usepackage{amsmath}
\usepackage{amsthm}
\usepackage{amstext}
\usepackage{amssymb}
\usepackage{latexsym}
\usepackage{dsfont}
\usepackage{makecell}

\usepackage{mathrsfs}
\usepackage{setspace}
\usepackage[normalem]{ulem}
\usepackage{bbold}
\usepackage{natbib}
\setcitestyle{authoryear,open={(},close={)}} %Citation-related commands

\DeclareMathAlphabet{\mathbb}{U}{msb}{m}{n}

\usepackage{hyperref}
\hypersetup{
    colorlinks=true,
    unicode=true
}

\usepackage{graphicx}
\usepackage{float}
\usepackage{caption, booktabs}
\usepackage{graphics}
\usepackage{subcaption}
\usepackage{lipsum}
\usepackage{amsmath,amssymb,amsthm,amsfonts}
\usepackage[table, dvipsnames]{xcolor}
\usepackage{mathrsfs}
\usepackage{soul}
\usepackage{mathtools}

\usepackage{threeparttable}
\newtheorem{theorem}{Theorem}
\newtheorem{prop}{Proposition}

\theoremstyle{definition}
\newtheorem{definition}{Definition}

\theoremstyle{remark}

\newcolumntype{H}{>{\setbox0=\hbox\bgroup}c<{\egroup}@{}}

\definecolor{forest}{HTML}{0003c2}

\definecolor{lila}{HTML}{6e00c2}

\definecolor{cyan}{rgb}{0,0.5,.5}
\definecolor{dcyan}{rgb}{0,0.3,.3}

\definecolor{leaf}{HTML}{009e1d}

\usepackage{tikz}
\usepackage{lineno}
\modulolinenumbers[0]

\begin{document}

% Comment out to disable line numbers
%\linenumbers

\begin{frontmatter}

    \journal{TBD
    }

    \title{The Field Equations of Penalized non-Parametric Regression} %\newline \textit{Discussion Paper}}

    %The Field Equations of Penalized Non-Parametric Regression

    \author[1]{Sven Pappert%\footnote{\noindent The author wants to express his gratitude for helpful discussions at the European Meeting of \hspace*{0.45cm} Statisticians 2023 by the Bernoulli Society and the Young Scientists Workshop 2023 by the \hspace*{0.45cm} German Statistical Society}
    \corref{cor1}}
    \ead{pappert@statistik.tu-dortmund.de}
    \cortext[cor1]{Corresponding author}
\address[1]{Chair of Econometrics, Department of Statistics, TU Dortmund University, Germany}

\begin{abstract}
We view penalized risks through the lens of the calculus of variations. We consider risks comprised of a fitness-term (e.g. MSE) and a gradient-based penalty. After establishing the Euler-Lagrange field equations as a systematic approach to finding minimizers of risks involving only first derivatives, we proceed to exemplify this approach to the MSE penalized by the integral over the squared $\ell^2$-norm of the gradient of the regression function. The minimizer of this risk is given as the solution to a second order inhomogeneous PDE, where the inhomogeneity is given as the conditional expectation of the target variable conditioned on the features. We discuss properties of the field equations and practical implications thereof, which also apply to the classical Ridge penalty for linear models, and embed our findings into the existing literature. In particular, we find that we can recover the Rudin-Osher-Fatemi model for image-denoising, if we consider the features as deterministic and evenly distributed. Last, we outline several directions for future research.
\end{abstract}
\end{frontmatter}
\section{Introduction}
\noindent
When training a non-parametric model w.r.t. a penalized risk, it is clear that the model is not estimating the conditional expectation. But what is the model estimating? In this work we want to introduce a perspective on penalized risks in non-parametric regression problems that takes into account this question and offers a way to answer it systematically. Answering this question is crucial for understanding the nature of penalization. To motivate the beforementioned perspective, we begin by considering a standard example, namely that the conditional expectation minimizes the MSE risk.
\subsection{Motivation}
\label{Introduction:Motivation}
\noindent
Assume we observe pairs of $iid$ random variables $(Y_i,X_i)$ with $i \in \{1,\hdots,n\}, X_i \in \Omega \subset \mathbb{R}^{p}$ with $\Omega$ being compact and $Y_i \in \mathbb{R}$. We assume that there is a relation between the $Y_i$ and $X_i$, i.e. there is a mapping $g_0: \Omega \rightarrow \mathbb{R}$ such that $Y_i = g_0(X_i) + \varepsilon_i$, where $\varepsilon_i \sim (0, \sigma^2)$ is an idiosynchratic, elliptic and $iid$ error term. The goal of non-parametric regression is to estimate the function $g_0$. By $\hat{g}_n$ we denote the estimator for $g_0$ based on $n$ observations. The parameter space, $\Theta_n$, containing the parameters which parametrize $g_n$ can vary with $n$ in accordance to the discussion in \cite{stone1982optimal}. A sensible $\hat{g}_n$ ought to be chosen in such a way that the distance between $g$ and $\hat{g}$ is small. Specifying this distance amounts to specifying an (empirical) risk $\hat{Q}: \mathbb{R}^n\times\Omega^{n}\times H(\Omega) \rightarrow \mathbb{R}$, where $H(\Omega)$ is a suitable function-space on $\Omega$. Choosing $\hat{Q}$ as the (scaled) empirical MSE risk and plugging in the target observations, $\mathbf{Y} := (Y_1,\hdots,Y_n) \in \mathbb{R}^n$ and a function $g \in \mathcal{L}^2(\Omega, \mathbb{P}^X)$ evaluated at the feature observations, $\mathbf{X} =(X_1,\hdots,X_n) \in \Omega^{n}$, i.e. $g(\mathbf{X}) := (g(X_1),\hdots,g(X_n))$, we may write 
\begin{align*}
\hat{Q}(\mathbf{Y},\mathbf{X};g) = \frac{1}{n}\sum_{i = 1}^{n} (Y_i - g(X_i))^2.
%\label{EmpRisk_MSE}
\end{align*}
Note that under suitable regulatory conditions, $\hat{Q}(\mathbf{Y}, \mathbf{X}; g)$ is a consistent and unbiased estimator for the MSE risk:
\begin{align}
Q[g] = \mathbb{E}[\left(Y_1 - g(X_1)\right)^2] = \int (y - g(x))^2 f_{Y,X}(y,x) dy dx.
\label{Risk_MSE}
\end{align}
The aim of this work is to find the function that minimizes risks as the one in Eq.~\ref{Risk_MSE} and risks which are penalized by a functional of the gradient of $g$. In order to do so, we view the risk $Q$ as a functional of the regression function $g$ and search for the $\tilde{g}$ that minimizes the risk functional. A systematic approach for finding extrema of functionals is offered by the calculus of variations. In Sect.~\ref{Basics:VariationalCalculus} we introduce some basic concepts, and refer to \cite{rindler2018calculus} for a comprehensive exposition.
%The calculus of variations offers a systematic approach to find an extremum of a functional. We refer to \cite{rindler2018calculus} for a comprehensive exposition of the calculus of variations. In Sect.~\ref{Basics:VariationalCalculus}, we also introduce some basic concepts.
%
% MSE Example
%
To give a simple illustration, we reiterate the result from Chapter 1.5.5 of \cite{bishop2006pattern}: The function minimizing the MSE-risk functional, $Q$ from Eq.~\ref{Risk_MSE}, is the mapping $x \mapsto \mathbb{E}(Y|X = x)$. The proof idea goes as follows: The first variation of $Q$ w.r.t. $g$ is derived and set to zero, $\delta Q[g] = 0$.\footnote{We suggest the readers who are not familiar with the calculus of variations to think of $\delta Q$ as the total differential of $Q$ w.r.t. a function instead of a variable for now. The differentiation w.r.t. a function is indicated by the '$\delta$' (instead of using a '$d$'.) In Sect.~\ref{Basics:VariationalCalculus}, we give a formal definition of the first variation.}
%footnote end
This yields an equation for $g$, whose solution is an extremum. The extremum actually being a minimum follows from the convexity of the MSE-risk. The complete derivation, which we adopt from \cite{bishop2006pattern} is as follows;
\begin{align}
\delta Q[g] =& \delta \int\int (y - g(x))^2 f_{Y,X}(y,x) dy dx \nonumber
\\
 =& \int\int 2(y - g(x)) \delta g(x) f_{Y,X}(y,x) dy dx \nonumber
\\
=& 2 \int \left[\int (y - g(x)) f_{Y,X}(y,x) dy\right] \delta g(x) dx \nonumber
\\
=& 2 \int \underbrace{\left[ \int y \frac{f_{Y,X}(y,x)}{f_{X}(x)} dy - g(x) \right]}_{= \mathbb{E}(Y|X = x) - g(x)} f_{X}(x) \delta g(x) dx. \nonumber
\end{align}
Setting $\delta Q[g] = 0$, the solution, denoted by $\tilde{g}$, is found as $\tilde{g}(x) = \mathbb{E}(Y|X = x)$ for all $x \in \Omega$.
Of course, this is not the only way to show that the conditional expectation minimizes the MSE-risk. However, this way offers a generalizable approach to systematically derive equations whose solutions are the minimizers of risk functionals, i.e. \textit{field equations}. This approach amounts to analyzing risks by understanding the risks as functionals of fields, and to find the minimum of the risk using the calculus of variations. The equations obtained by requiring $\delta Q[g] = 0$ are the field equations (the naming is in analogy to physics). A similar derivation as for the MSE-risk can be conducted for the MAE Risk. The results are given in the first two rows of Table~\ref{Table1}.\\
%
% Motivation for penalized regression
%
Notice that the unpenalized risk functionals in the first two rows of Table~\ref{Table1} depend only on the field $g$ itself. Given a certain penalty, e.g. a Ridge-penalty, the functional form of the penalty term on the underlying regression function
%, i.e. how you would write $\sum_i^p\beta_i^2$ in terms of $g$,
is not obvious, i.e. how would you write $\sum_i^p\beta_i^2$ in terms of the regression function $g$? It is also not obvious what the corresponding field equation and its respective solution is. To emphasize these points, the corresponding cells in the table are just question marks, (?1)-(?3), and the penalty is just written as $\widehat{\text{Pen}}[g]$.
The goal of this work is to fill these cells for the case where the penalty is gradient-based.\footnote{This already answers the question 'how would you write $\sum_i^p\beta_i^2$ in terms of the regression function $g$?' for linear $g$. The Ridge-penalty can be obtained by taking the integral of the squared $\ell^2$-norm of the gradient of $g$, $\int ||\nabla g(x)||_2^2 d\mathbb{P}^{X}(x) = \sum_{i = 1}^p \beta_i^2$. For non-linear $g$. We will further discuss this in Sect.~\ref{Basics:PenalizedRegression:NonParametric}.}
In the related literature different choices for the penalty are available. Taking the integral of the $\ell^1$ or $\ell^2$ norm of $\nabla g$ (or versions thereof) are popular choices, see e.g. \cite{drucker1991improving, drucker1992improving, rosasco2013nonparametric, follain2024nonparametric, ding2024semi} or \cite{shen2022consistency}. In Sect.~\ref{Basics:PenalizedRegression:NonParametric} we give an overview over the different approaches and distinguish between the respective risk and its empirical counterpart. The (?1)-cell will then be filled by what we will define as \textit{risk functional}, see Def.~\ref{Def:FirstRiskFunctional}.
%and argue show that for linear $g$, the (squared) $\ell^2$-penalty yields the Ridge-penalty, cf. \cite{hoerl1970ridge} and \cite{hastie2017elements}, Sect. 3.4.1 and the $\ell^1$-penalty yields the LASSO-penalty, cf. \cite{tibshirani1996regression}.
In order to fill the (?2)-cell,we take the first variation of the risk functional. In particular, if the penalty only depends on the first derivative, the field equations are given by the Euler-Lagrange equations, cf. proposition~\ref{prop:FirstOrderRiskFunctionalFieldEqs} of this paper or \cite{rindler2018calculus}, Theorem 3.1. Filling the last cell of the table, (?3), amounts to solving partial differential equations (PDEs). For general cases no closed form solutions are available, however there are special cases where the resulting differential equation can be solved. Some of these are outlined in Sect.~\ref{FieldEqs:squared-ell2}. \\
\subsection{Contribution}
\noindent
The contribution of this paper can be structured in three parts. First, we define the risk functional in the non-parametric regression context and integrate gradient-based penalization from the related literature into this framework. Second, we derive the minimum of a first-order risk functional as the solution to the Euler-Lagrange equations associated to the risk functional. We investigate the special case where the gradient-based penalty is the integral of the squared $\ell^2$-norm of the gradient of the regression function. Third, we relate our framework to existing literature on image processing. In particular, we see that we obtain the Rudin-Osher-Fatemi (ROF) framework,  cf. \cite{rudin1992nonlinear}, under the condition that the features deterministic and evenly distributed. The ROF framework was introduced by the authors in order to remove noise from images, there the 'features' are just the places of a pixel. In our framework, we allow for random features following any (continuous) distribution. Consequently the joint distribution of the features influences the risk functional. 
%This explains why it corresponds to our framework when we set $f_X = \frac{1}{\lambda(\Omega)} \mathds{1}_{\Omega}$, because the place of a pixel is fixed. %
\subsection{Overview}
\label{Introduction:Overview}
\noindent
In Sect.~\ref{Basics}, we gintroduce the (basic) concepts, which are necessary to understand this work. These include the calculus of variations, Sect.~\ref{Basics:VariationalCalculus}, as well as linear and non-parametric penalized regression, Sect.~\ref{Basics:PenalizedRegression:Linear} and Sect.~\ref{Basics:PenalizedRegression:NonParametric}, respectively. For the penalized non-parametric regression, we focus on gradient-based penalization. We give examples for gradient-based penalization from the related literature, including deep learning, non-parametric regression and image denoising, and embed them in our definition of a risk functional, Def.~\ref{Def:FirstRiskFunctional}. In Sect.~\ref{prop:FirstOrderRiskFunctionalFieldEqs} we derive the field equations for risk functionals, represented as the sum of a fitness-term and a gradient-based penalty only involving first derivatives. In Sect.~\ref{FieldEqs:squared-ell2}, we proceed to analyze the special case when the risk functional comprises the MSE-risk and the squared-$\ell^2$-penalty. This is the non-parametric equivalent to the Ridge penalty.
%In Sect.~\ref{FieldEqs:ell1}, we analyze the case, when the MSE is penalized by an $\ell^1$-penalty, which is the non-parametric generalization of the LASSO-penalty.
In Sect.~\ref{Discussion}, we discuss our findings, their restrictions, possible applications and further research questions.
%In Chapter 5 we discuss the insights gained from this work and possible future avenues of research. Especially we focus on the possibility of incorporating the solutions to the field equations in a saddle-point approximation of functional integrals arising in the statistical functional analysis of univeral function approximators.
%
\begin{table}[t]
%\begin{center}
\caption{Empirical risks, risks, their field equations and solutions}
\label{Table1}
\centering
%\resizebox{\columnwidth}{!}{%
%\label{tab1}
\begin{tabular}{c|c|c|c|c}
\hline\noalign{\smallskip}
Name & empirical risk & risk & Field Eqs. & Solution
\\
\hline
%\noalign{\smallskip}\svhline\noalign{\smallskip}
MSE & $\frac{1}{n}\sum_{i=1}^{n}(y_i - g(\mathbf{x}_i))^2$ & $\mathbb{E}((Y - g(\mathbf{X}))^2)$ & $0 = \mathbb{E}(Y|\mathbf{X} = \mathbf{x}) - g(\mathbf{x})$ & $g = \mathbb{E}(Y|\mathbf{X})$
\\
\hline
MAE & $\frac{1}{n}\sum_{i=1}^{n}|y_i - g(\mathbf{x}_i)|$ & $\mathbb{E}(|Y - g(\mathbf{X})|)$ & $0 = g(\mathbf{x}) - F_{Y|\mathbf{X} = \mathbf{x}}^{-1}(0.5)$ & $g(\mathbf{x}) = F_{Y|\mathbf{X} = \mathbf{x}}^{-1}(0.5)$
\\
%\hline
%Energy Distance & \makecell{$\frac{2}{n^2}\sum_{i = 1}^n\sum_{j = 1}^n |y_i - g(x_i)|$ \\ $- \frac{1}{n^2}\sum_{i = 1}^n\sum_{j = 1}^n |y_i - y_j|$ \\ $- \frac{1}{n^2}\sum_{i = 1}^n\sum_{j = 1}^n |g(x_i) - g(x_j)|$} & - & - & -
%\\ 
\hline
Penalized-MSE & \makecell{$\frac{1}{n}\sum_{i=1}^{n}(y_i - g(\mathbf{x}_i))^2$ \\ $ + \lambda \widehat{\text{Pen}}[g]$} & (?1) & (?2) & (?3)
\\
\hline 
\end{tabular}%
%}
%\end{center}
\end{table} 

\section{Basics}
\label{Basics}

\subsection{Calculus of Variations}
\label{Basics:VariationalCalculus}
\noindent
The goal of the calculus of variations is to search for minimizers of a functional in a predetermined space of functions. The space of functions considered in this paper is the Sobolev space on $\Omega \subset \mathbb{R}^p$. As is common, we denote by $W^{k,p}(\Omega)$ the Sobolev space, where $L^p$-integrability of up to the $k$-th derivative is required, cf. \cite{rindler2018calculus}, Appendix A5.
We note that integrability w.r.t. a non-singular measure $\mu$ is implied.\\
The functionals, whose minima we are interested in, are mappings from a Sobolev space to the real numbers, $\Gamma: W^{k,p}(\Omega) \rightarrow \mathbb{R}$. 
% Euler-Lagrange
A major result of the calculus of variations are the well-known Euler-Lagrange equations, which we present in Thm.~\ref{Thm:EulerLagrangeEquations} and make use of in Sect.~\ref{Section:FieldEqsOfPenNonParRef}. Minimzers of functionals of the form $\Gamma[\phi] = \int f(x, \phi(x), \nabla \phi(x))$ are subject to the Euler-Lagrange equation. We state this well-known result in the following theorem and refer to Theorem 3.1 and Proposition 3.3 from \cite{rindler2018calculus} for a formulation with predetermined boundary conditions.\footnote{Theorem 3.1 assures that minimizers of $\Gamma$ are subject to the Euler-Lagrange equations while proposition 3.3 assures that solutions to the Euler-Lagrange equations are also minimizers of the functional.}
Alternatively, the first theorem in Sect. 2.2 from \cite{giaquinta1996calculus} may be referenced for a formulation without predetermined boundary conditions. Furthermore, we require regulatory conditions, which are given in \cite{rindler2018calculus}. For completeness, we restate them in~\ref{App:AssumptionsEulerLagrange}.
\begin{theorem}[Euler-Lagrange Equations]
\label{Thm:EulerLagrangeEquations}
Let $\Gamma: W^{2,p}(\Omega) \rightarrow \mathbb{R}$ be given as $\Gamma[\phi] = \int f(x, \phi(x), \nabla \phi(x)) dx$. Furthermore let Assumptions A1)-A3) from~\ref{App:AssumptionsEulerLagrange} be fulfilled. If $\phi \in  W^{1,p}(\Omega)$ minimizes $\Gamma$, then $\phi$ is subject to the following boundary problem
\begin{align}
\begin{cases}
0 = - \nabla \frac{\partial f(x, \phi(x), \nabla \phi(x))}{\partial \nabla \phi(x)} + \frac{\partial f}{\partial \phi(x)}(x, \phi(x), \nabla \phi(x)), & x \in \Omega,
\\
0 = \frac{\partial f}{\partial \nabla \phi} \cdot \hat{n}, & x \in \partial \Omega,
\end{cases}
\label{Eq:EulerLagrangeFields}
\end{align}
where $\hat{n}$ denotes the normal vector of $\Omega$. 
\end{theorem}
\textbf{Remarks:}
\begin{itemize}
%\item[a)] Even though we hope we made it clear, we would like to emphasize that this result is by no means new. It is well known and given in books about the calculus of variations. The contribution of this paper is the application of the Euler-Lagrange equations to risk functionals.
\item[a)] In Theorem 3.1 from \cite{rindler2018calculus}, the boundary condition is different from the one we give here. The reason for this is that we do not have requirements on $\phi$ on the boundary ab initio. Instead, we adopt the \textit{natural boundary conditions}, cf. \cite{giaquinta1996calculus}, Sect. 2.4.
\item[b)] Theorem~\ref{Thm:EulerLagrangeEquations} does not assure that the solution to the Euler-Lagrange equation is a minimum. The conditions under which a solution to the Euler-Lagrange equation is a minimum are part of Prop. 3.3 from \cite{rindler2018calculus}. For completeness, we restate this proposition as Prop.~\ref{prop:ConditionsELSolutionMinimum} below.
\end{itemize}
\begin{prop}[Conditions for solution to the E.-L.-Eqs. to be a minimum, from \cite{rindler2018calculus}, Prop. 3.3]
\label{prop:ConditionsELSolutionMinimum}
Let $\Gamma: W^{2,p}(\Omega) \rightarrow \mathbb{R}$ be given as $\Gamma[\phi] = \int f(x, \phi(x), \nabla \phi(x)) dx$. If the following conditions hold, the solution to the Euler-Lagrange equation of $\Gamma$ is a minimum of $\Gamma$: Let $x \in \Omega$, $s \in \mathbb{R}$ and $u \in \mathbb{R}^p$.
\begin{itemize}
\item[i)] $|\frac{\partial}{\partial_s} f(x,s,u)| \leq C (1 + |s|^{p-1} + ||u||^{p-1})$ and $||\frac{\partial}{\partial_u} f(x,s,u)|| \leq C (1 + |s|^{p-1} + ||u||^{p-1})$ for $C > 0$ and $p > 1$.
\item[ii)] For fixed $x$, $f(x,s,u)$ is convex in $(s,u)$.
\end{itemize}
\end{prop}
We note that the proposition in \cite{rindler2018calculus} also holds for vector-valued fields. Since it is sufficient for our case, we only gave the one-dimensional version of the proposition. The $x, s$ and $u$ are placeholders for $x, \phi(x), \nabla \phi(x)$. \cite{rindler2018calculus} named these placeholders $x,v$ and $A$, respectively.\\
The Euler-Lagrange equations can be derived by taking the first variation of the functional of interest. This is also what we do in the motivation of this paper, Sect.~\ref{Introduction:Motivation}. The first variation of a functional is defined as
\begin{align}
\delta \Gamma[\phi] = \lim_{\varepsilon \rightarrow 0} \frac{ \Gamma[\phi + \varepsilon \eta] - \Gamma[\phi]}{\varepsilon}.
\end{align}
where $\eta$ is a test function on which we enforce the same assumptions as on $\phi$ \citep[cf.][Eq. 3.1]{rindler2018calculus}.
This expression is similar to the usual differential quotient. The difference is that it is not a function whose value changes by changing the argument by a small amount, rather it is a functional which is changed by changing its argument by adding a test function with a small coefficient. The Euler-Lagrange equations are the result of setting
\begin{align*}
\delta \Gamma[\phi] = 0
\end{align*}
for $\Gamma$ as in Theorem.~\ref{Thm:EulerLagrangeEquations}. This procedure is analogous to setting the derivative (or the gradient) of a function to zero to find its extrema.
Now, the other main concept of this paper -- Penalized Regression -- is introduced. 
\subsection{Penalized Regression}
\label{Basics:PenalizedRegression}
\noindent
Let $(Y_i,\mathbf{X}_i), i \in \{1,\hdots,n\}$ be pairs of $iid$ and integrable random variables/vectors. $Y_i \in \mathbb{R}$ is the (continuous) target variable, while $\mathbf{X}_i \in \Omega \subset \mathbb{R}^p$ are (continuous) features. It is assumed that there exist square-integrable, $iid$ and elliptic random variables $\varepsilon_i \sim (0, \sigma^2)$ and a function $g_0:\Omega \rightarrow \mathbb{R}$ s.t. we may write $Y_i = g_0(\mathbf{X}_i) + \varepsilon_i$. Now we give a brief to penalized linear regression, followed by the the non-parametric case afterwards.
\subsubsection{Penalized Regression for Linear Models}
\label{Basics:PenalizedRegression:Linear}
\noindent
Assuming that $g_0$ is linear, i.e. we may write $g_0(x) = \gamma_0 + \sum_{i = 1}^p \gamma_i x_i, \ \forall x \in \Omega$, we take the regression function $g$ to be linear as well, $g(x) = \beta_0 + \sum_{i = 1}^p \beta_i x_i$. Then we can write the penalized empirical scaled MSE risk, evaluated at observations $\mathbf{Y}=(Y_1,\hdots,Y_n)$ and $\mathbf{X} = (\mathbf{X}_1,\hdots,\mathbf{X}_n)$, respectively, and parameters $\beta$ as
\begin{align}
\hat{Q}(\mathbf{Y}, \mathbf{X}; \beta) = \frac{1}{n} \sum_{i = 1}^n \left(Y_i - (\beta_0 + \sum_{j = 1}^p \beta_i X_{ij})\right)^2 + \lambda \widehat{\text{Pen}}(\beta) = \frac{1}{n} (\mathbf{Y} - \mathbf{X}\beta)^2 + \lambda \widehat{\text{Pen}}(\beta).
\label{LinearPenalizedMSE}
\end{align}
where $\lambda > 0$ is the penalization parameter and $\widehat{\text{Pen}}(\beta)$ is a suitable penalty, e.g. the Ridge-penalty, $\widehat{\text{Pen}}(\beta) = \sum_{i = 1}^p \beta_i^2$. The idea is to penalize deviations from constancy s.t. simplicity is favored. 
Notice that the penalty depends on a vector of parameters, $\beta$, and not on a field as in the considerations for non-parametric regression in the introduction (and in the next subsection, Sect.~\ref{Basics:PenalizedRegression:NonParametric}). The theoretical counterpart to Eq.~\ref{LinearPenalizedMSE}, i.e. the penalized MSE-risk for linear functions, evaluated at parameters $\beta = (\beta_0, \hdots, \beta_p)$, is given as
\begin{align*}
Q(\beta) = \mathbb{E}\left(Y_1 - (\beta_0 + \sum_{j = 1}^p \beta_i X_{1j})^2\right) + \lambda \widehat{\text{Pen}}(\beta).
\end{align*}
Since $\widehat{\text{Pen}}(\beta)$ only depends on deterministic $\beta$, the penalty takes the same form in the empirical and the non-empirical risk. This is a major difference to non-parametric penalization.
Depending on the choice of the penalty, the empirical risk and the respective estimators have different properties. We discuss these properties in the following and thereby introduce different penalties. \\
% No Penalization
We start with the case where no penalization is present, i.e. $\lambda\widehat{\text{Pen}}(\beta) = 0, \forall \beta \in \mathbb{R}^{p}$. With no penalization, we are in the OLS setting and under regularity conditions, the Gauss-Markov theorem assures that the OLS estimator, $\hat{\beta}^{\text{OLS}} = (\mathbf{X}^T\mathbf{X})^{-1}\mathbf{X}^T \mathbf{Y}$, is the best linear unbiased estimator, called 'BLUE'.\footnote{See Chapter 1.3 of \cite{{hayashi2011econometrics}} for the Gauss-Markov theorem and the BLUE property}.
While theoretically being the best estimator, the OLS-estimator has some problems in certain set-ups. First, \cite{hoerl1970ridge} point out that, if the features are highly correlated, the estimator is not stable, i.e. it has a high variance. Second, \cite{tibshirani1996regression} also points out that the OLS estimator is subject to high variance and that the estimator suffers from estimates deviating from their true value, especially zero, just by chance. These problems become especially relevant in high-dimensional settings. Attempts to overcome these problems lead us to the Ridge and LASSO regularization methods, discussed next.\\
%Ridge Regression
%After realizing that the OLS estimator may be unstable if $(\mathbf{X}^T\mathbf{X})$ strongly deviates from the unity matrix,
\cite{hoerl1970ridge} introduce the Ridge estimator as $\beta^{\text{Ridge}} = (\mathbf{X}^T\mathbf{X} + \lambda I_{p\times p})^{-1} \mathbf{X}^T \mathbf{Y}$. The BLUE property of the OLS-estimator is sacrificed in order to obtain a more stable estimator. The relation of $\hat{\beta}^{\text{Ridge}}$ to the empirical loss in Eq.~\ref{LinearPenalizedMSE} is that $\hat{\beta}^{\text{Ridge}}$ minimizes $\hat{Q}$ when $\widehat{\text{Pen}}(\beta)$ is chosen as $\widehat{\text{Pen}}(\beta) = \sum_{i =1}^p \beta_i^2$ \citep[see Chapter 3.4.1 of][]{hastie2017elements}. Consequently, the term $\sum_{i =1}^p \beta_i^2$ is called Ridge-penalty. An interesting property of the Ridge estimator is that $\mathbb{P}(\beta^{\text{Ridge}}_i = 0) = 0$, i.e. the probability that parameters are shrinked to zero exactly is zero.
%This is due to the decreasing absolute value of the slope near the origin of the square function [ELEMENTS?].
To this end, the LASSO (least absolute shrinkage and selection operator) penalty was introduced by \cite{tibshirani1996regression}. Using the LASSO, shrinking parameters to exactly zero becomes possible. The LASSO is given as $\widehat{\text{Pen}}(\beta) = \sum_{i = 1}^p |\beta_i|$. Due to the absolute value function not being differentiable at $0$, we will not analyze penalties of this form in the current framework. However, in the discussion, Sect.~\ref{Discussion:OtherPenalties}, we pick up penalties involving absolute values again and sketch a way to find minima of the corresponding risk functionals.
\subsubsection{Gradient-based Penalization in non-Parametric Regression}
\label{Basics:PenalizedRegression:NonParametric}
\noindent
Before discussing approaches from the gradient-based penalization literature, we want to give a general form of an (empirical) risk which is to be analyzed by the calculus of variations. Later in this section we will embed these approaches into our general framework. The risk functional we are interested in is the sum of two terms; the first term assesses a model's fit, while the second term is the penalty term.
We denote the fitness term by $L$ and the penalty by $\text{Pen}$, with their empirical counterparts being $\hat{L}$ and $\widehat{\text{Pen}}$, respectively.
%We may call the (empirical) fit-term $\hat{L}$ or $L$, respectively, and the (empirical) Penalty-term $\widehat{\text{Pen}}$ or $\text{Pen}$, respectively.
Examples for $L$ are the MSE, (log-)Likelihood, MAE or the Energy Distance. Examples for $\text{Pen}$ are integrals of $\ell^2$ or $\ell^1$ norms (or versions thereof) of the gradient. 
\begin{definition}[Risk Functional]
\label{Def:FirstRiskFunctional}
Let $\widehat{\text{Pen}}: W^{k,p}(\Omega) \times \Omega \rightarrow \mathbb{R}$ and $\hat{L}:W^{k,p}(\Omega) \times \Omega \times \mathbb{R} \rightarrow \mathbb{R}$. The empirical risk functional evaluated at $\mathbf{Y} \in \mathbb{R}$, $\mathbf{X}\in \Omega$ and $g \in W^{k,p}(\Omega)$, is given as
\begin{align}
\hat{Q}(\mathbf{Y}, \mathbf{X}, g) = \frac{1}{n}\sum_{i=1}^n \hat{L}(Y_i,g(X_i)) + \lambda\widehat{\text{Pen}}(g, \mathbf{X}).
\end{align}
The penalty may include partial derivatives of $g$ up to order $(k-1)$. The fitness term may only include $g$ itself.
The theoretical counterpart to the empirical risk functional, evaluated at $g \in W^{k,p}(\Omega)$, is given as $Q: W^{k,p}(\Omega) \rightarrow \mathbb{R}$ such that
\begin{align}
Q[g] = \int \hat{L}(y, g(x)) d\mathbb{P}^{(X,Y)}(x,y) + \lambda \int \widehat{\text{Pen}}(g, x) d\nu(x)
\label{RiskFunctional}
\end{align}
where $\mathbb{P}^{(X,Y)}$ is the image-measure of the random vector $(X,Y)$ and $\nu$ is a measure on $(\Omega, \mathcal{B}(\Omega))$. Both measures are assumed to be $\sigma$-finite measures which are absolutely continuous w.r.t. the Lebesgue-measure on the respective space. The Radon-Nikodym Lebesgue densities may be called $f_{X,Y}$ and $f_{\nu}$.
\end{definition}
\noindent
\subsubsection*{Remarks 1}  
\begin{itemize}
%\item[a)] Note that we have written that we evaluate $\hat{Q}$ at $\mathbf{Y}$, $\mathbf{X}$ and $g$ and not at $\mathbf{Y}$ and $g(\mathbf{X})$. This is because now we consider the empirical penalty as a mapping, $\widehat{\text{Pen}}: W^{1,p}(\Omega) \times \Omega \rightarrow \mathbb{R}$, which, in particular, is a functional in $\nabla g$, and not $g$
\item[a)] The measure $\nu$ does not need to be a probability measure nor does it need to be the measure associated to the random variable $X$.
\item[b)] Many gradient-based penalization approaches from the literature are related to the risk functional we define here. Def.~\ref{Def:FirstRiskFunctional} is an attempt to 1) generalize risks with gradient-based penalties in such a way that popular approaches from the literature are included and 2) have the risk in a form which allows for analysis by variational calculus.
\item[d)] Note that $g \in W^{k,p}$ and that we allow up to $(k-1)$-order partial derivatives in the penalty term. This is due to the need for integration by parts when taking the first variation of the penalty. Furthermore, $g \in W^{k,p}$ ensures $L^p$-integrability w.r.t. the Lebesgue-measure of up to the $k$-th derivative. In the risk functional, we integrate w.r.t. $\sigma$-finite measures which are absolutely continous w.r.t. the Lebesgue-measure. The existence of the integrals is guaranteed by the Lebesgue-integrability of $g$.
\end{itemize}
\noindent
Now we proceed to give examples gradient-based penalization from the literature and, when possible, embed them into the risk functional framework from Def.~\ref{Def:FirstRiskFunctional}.\\
% ell^2-squared
First, we consider $\ell^2$-based penalties and start with the squared-$\ell^2$-penalty:
\begin{align}
\widehat{\text{Pen}}(\nabla g) =& \frac{1}{n} \sum_{i = 1}^n ||\nabla g(X_i)||_2^2,
\\
\text{Pen}(\nabla g(X_1)) =& \int ||\nabla g(x)||_2^2 d\nu,
\label{NonParRidgePenalty}
\end{align}
respectively.
This penalty is sometimes called 'Sobolev-regularization' (see e.g. \cite{rosasco2013nonparametric}, remark 1; \cite{ding2024semi}, Sect. 1). \cite{ding2024semi} also consider a penalty of the form $\frac{1}{n} \sum_{i = 1}^n ||\nabla g(Z_i)||_2^2$, where $Z_i$ are $iid$ random variables on $\Omega$ which do not have to correspond to the measure associated to $X$ (see Eq. 3 of their paper). This of course amounts to specifying the measure $\nu$ in Eq.~\ref{NonParRidgePenalty} accordingly.
%TODO Hier Wahba spezifizieren
When $g$ is linear, i.e. we may write $g(x) = \beta_0 + \sum_{i = 1}^p \beta_i x_i, \ \forall x \in \Omega$, then $\widehat{\text{Pen}}(\nabla g) = \text{Pen}(\nabla g(X_1)) = \sum_{i = 1}^p \beta_i^2$. This is the classical ridge penalty corresponding to the ridge estimator from \cite{hoerl1970ridge}, compare also Sect.~\ref{Basics:PenalizedRegression:Linear}. The embedding into the risk functional framework from Def.~\ref{Def:FirstRiskFunctional} is obvious: $\text{Pen}:W^{2,2} \rightarrow \mathbb{R}$ is only dependent on first-order partial derivatives. Risks of such form may be called \textit{First-order risk functionals}. The penalty in smoothing spline models is also of squared-$\ell^2$-form. There, the form of the penalty is $\int_{\Omega} (g^{(m)}(x))^2 dx$, where $g^{(m)}$ denotes the $m$-th order derivative, cf. \cite{wahba1990spline}, Chapter 1. Choosing $\widehat{\text{Pen}}(g,X) = g^{(m)}$ and $\nu = \lambda$, we may embed the smoothing spline risk into the risk functional framework too. Risks with $m$-th order derivatives may be called \textit{$m$-th-order risk functionals}. Additionally, we recognize that the empirical fitness term for a smoothing spline is $\frac{1}{n} \sum_{i = 1}^n (Y_i - g(x_i))^2$. Only $\{Y_i\}_{i = 1}^n$ are considered random variables, while $\{x_i\}_{i = 1}^n$ are considered to be deterministic. Accordingly, the image-measure $\mathbb{P}^{(X,Y)}$ in Eq.~\ref{RiskFunctional} can be written as $\mathbb{P}^{(X,Y)} =\frac{\#^{p}}{n} \otimes \mathbb{P}_Y$ to recover the smoothing spline risk.
\\
A related penalty is introduced in \cite{drucker1991improving, drucker1992improving}. There, the 'backward energy function', $\hat{E}_b = \sum_{i = 1}^n ||\nabla \left( \frac{1}{2} (Y_i - g(X_i))^2\right)||_2^2$, is used as penalty in the context of artificial neural networks (ANNs). The authors introduce double-backpropagation to calculate $\hat{E}_b$ efficiently when $g$ is an ANN. The backward energy function can be embedded into the risk functional framework by writing $E_b$ = $\int ||\frac{1}{2}\nabla (y - g(x))^2||_2^2 d\mathbb{P}^{(X,Y)}(x,y)$ and $\text{Pen}: W^{2,2} \rightarrow \mathbb{R}$ with $\text{Pen} = E_b$. Then, the penalty depends $g$ and its first partial derivatives.\footnote{In the original paper, the notation is different and the risk and penalty are given for multiple output networks. We adjusted the notation and dimensionality to fit in the context of this paper.}
All before mentioned penalties are squared $\ell^2$-based penalties. Without squaring, the penalty involves the euclidean norm. Such penalties are popular in image denoising. In the seminal paper by \cite{rudin1992nonlinear}, the authors use the penalty $\int ||g(\mathbf{x})||_2 d\mathbf{x}$.\footnote{In signal denoising the fidelity term has a different structure and does not include a density. In the paper, the authors}
In the context of signal-denoising this penalty is often called 'total variation'; amongst others it is used in \cite{torres2014total} for gravitational wave denoising. Similar to the penalization in the spline risk, the penalty does not involve random variables. 
\\
%
% ell^1
Next, we discuss the $\ell^1$-penalty:
\begin{align}
\widehat{\text{Pen}}(\nabla g, \mathbf{X}) =& \frac{1}{n}\sum_{i = 1}^n||\nabla g(X_i)||_1,
\\
\text{Pen}(\nabla g) =& \int ||\nabla g||_1 d \nu.
\end{align}
A penalty of this form is used in \cite{shen2022consistency} for penalization of ANNs with the rectified linear unit (ReLu) activation function. It is clear that for a linear function $g$, the classical LASSO penalty from \cite{tibshirani1996regression} is recovered. For the case where $g$ is parametrized as a GAMLSS, the penalty from \cite{groll2019lasso} is recovered. \\
%TODO Überprüfen ob das wirklich stimmt.
Last, we want to introduce the RKHS \& gradient-based penalization by \cite{rosasco2013nonparametric}, which is introduced for general purpose regularization. In the paper, the authors analyze an empirical risk of the following form:
\begin{align*}
\frac{1}{n}\sum_{i = 1}^n (Y_i - g(X_i))^2 + \lambda\left(2 \sum_{a = 1}^p \bigg|\bigg| \frac{\partial g}{\partial x_a}\bigg|\bigg|_n + v ||g||_{\mathcal{H}}^2\right).
\end{align*}
The first part of the empirical risk is just the usual MSE. The second part, the penalty, consists of a derivative based term, $\sum_{a = 1}^p \big|\big| \frac{\partial g}{\partial x^a}\big|\big|_n = \sum_{a = i}^p \sqrt{\frac{1}{n} \sum_{i = 1}^n \big(\frac{\partial g(X_i)}{\partial x_a}\big)^2} $, and a RKHS-penalty, $v ||f||_{\mathcal{H}}^2$. The former is based on squared derivatives and shares resemblance with the group LASSO, see \cite{yuan2006model} for the group LASSO. The term $v ||g||_{\mathcal{H}}^2$ is a norm on the Hilbert space consisting of functions which fulfill the reproducing property w.r.t. a pre-determined kernel $K$ (see \cite{rosasco2013nonparametric}, Sect. 4.) The norm depends on the explicit choice of the function space.
%If the norm is the $L^2$-norm, we can write $\hat{L} = \text{MSE} + \lambda v ||\cdot||_{L^2}^2$ and $\widehat{\text{Pen}}((\mathbf{h}(x_1), \hdots, \mathbf{h}(x_n))) = 2\sum_{a = i}^p \sqrt{\frac{1}{n} \sum_{i = 1}^n h_a(x_i)^2}$ to embed the empirical risk into our framework.
The non-empirical risk can be obtained by replacing the MSE by the expected MSE and $\frac{1}{n}\sum_{i = 1}^n \left(\frac{\partial g(X_i)}{\partial x_a}\right)^2$ by $\mathbb{E}[(\frac{\partial g}{\partial x_a}(X))^2]$, given that the corresponding conditions for convergence are fulfilled. The embedding into the risk functional framework depends on the explicit choice of the kernel $K$. In its general form, the risk can not be embedded into the risk functional.\\
Another gradient-based penalization approach, which can not be represented as a risk functional is the one by \cite{follain2024nonparametric}. There, an $\ell^2$-type penalty is used but the authors focus on Hermite polynomials and introduce coefficients for each summand of the Hermite polynomial series. The introduction of the coefficient hinders the analysis with our approach.   
%The term $\lambda v ||\cdot||_{L^2}^2$ is deterministic and remains the same in the non-empirical setting. 
Of course, there are much more regularization approaches available in the literature and this is not an exhaustive discussion. Here, we merely introduced what we think are the basic building blocks of gradient-based penalization. In the next section, Sect.~\ref{Section:FieldEqsOfPenNonParRef}, we proceed to derive the field equations of the a first-order risk functional.
\section{Field equations of Penalized Non-Parametric Regression}
\label{Section:FieldEqsOfPenNonParRef}
\noindent
Assume the set-up from above, i.e. a non-parametric regression setting with a risk penalized by a gradient-based penalty. In this section we analyze penalized risks by means of the calculus of variations. We start by proposing the general field equations for first-order risk functionals, i.e. penalties only involving differentiable mappings of the first gradient. Then we proceed by explicitly stating the field equations of the squared-$\ell^2$ penalized MSE-risk.
\begin{prop}
\label{prop:FirstOrderRiskFunctionalFieldEqs}
Let $(Y_i,X_i)_{i = 1}^n$ be $iid$ pairs of random variables and $Y_i \in \mathbb{R}$, $X_i \in \Omega \subset \mathbb{R}^p$. Let $\mathbb{P}^X$ be the image-measure associated to the random variables $X_i$, $\mathbb{P}^Y$ the image measure associated to the random variables $Y_i$ and $\mathbb{P}^{(X,Y)}$ their joint image-measure. Furthermore, let $\nu$ be a measure on $(\Omega, \mathcal{B}(\Omega))$ s.t. $\text{Pen}(\nabla g) \in \mathcal{L}(\Omega, \nu)$. All measures are assumed to be $\sigma$-finite and absolutely continuous w.r.t. the respective Lebesgue-measure on the space and their Lebesgue-densities are denoted as $f_X$, $f_Y$, $f_{X,Y}$ and $f_{\nu}$, respectively. Let $g \in W^{2,p}(\Omega)$ and let $Q$ be a risk functional of the form
\begin{align}
Q[g] = \int L[y; g(x)]d\mathbb{P}^{(X,Y)}  + \lambda \int \text{Pen}[\nabla g] d\nu. 
\end{align}
Furthermore, let assumptions (A1)-(A5) from~\ref{App:AssumptionsEulerLagrange} hold. Then $Q$ is minimized by the solution to the following differential equation (Euler-Lagrange equation):
\begin{align}
0 = \int \frac{\partial L(x, y, g)}{\partial g} f_{X,Y}(x,y) dy - \lambda \, \nabla \left[ \frac{\partial\text{Pen}}{\partial\nabla g}[\nabla g(x)] f_{\nu}(x) \right],
\label{Eq:FirstOrderRiskFunctionalFieldEqs}
\end{align}
subject to boundary conditions: $\frac{\partial \text{Pen}}{\partial \hat{\mathbf{n}}}[g] = 0$ on $\partial \Omega$, where $\frac{\partial}{\partial \hat{\mathbf{n}}}$ denotes the normal derivative.
\end{prop}
\begin{proof}
We write $Q$ as
\begin{align}
Q[g] =& \int f(x, g(x), \nabla g(x)) dx, \ \text{where}
\\
f(x, s, u) =& \int L(y,s) f_{X,Y}(x,y) dy + \lambda \text{Pen}(u) f_{\nu}(x).
\end{align}
Then the minimizer of $Q$ is given as the solution to the Euler-Lagrange equations of the system, see Theorem.~\ref{Thm:EulerLagrangeEquations}. Since no boundary conditions are imposed a priori, the natural boundary conditions are adopted.
\end{proof}
\textbf{Remarks:}
\begin{itemize}
\item[a)] We note that the field equations in Eq.~\ref{Eq:FirstOrderRiskFunctionalFieldEqs} are the Euler-Lagrange equations of the first-order risk functional. The difference to usual use-cases of the Euler-Lagrange equations, e.g. in physics, is that the integrals here are measure integrals and, as a consequence, the respective Lebesgue-densities enter the equation. 
\item[b)] We take $g(0) = \mathbb{E}(Y|X = 0)$ as auxiliary boundary condition when necessary. The necessity occurs e.g. in the limiting case $\lambda \rightarrow \infty$. This particular choice will be motivated in the following example of the squared-$\ell^2$ penalized MSE-risk. It would also be possible to impose boundary conditions such as $g = \mathbb{E}(Y|X = \cdot)$ on $\partial \Omega$ a priori. Imposing no boundary conditions a priori is an explicit choice that we make.
\end{itemize}
Now we explicitly give the field equations for the squared-$\ell^2$ penalized MSE-risk.
\subsection{Field Equations of the squared-$\ell^2$ Penalized MSE-Risk}
\label{FieldEqs:squared-ell2}
\noindent
Let $\Omega = [0,1]^p, L(x,y) = (x - y)^2$ and $\text{Pen}(w) = ||w||_2^2$ and $\nu = \mathbb{P}^X$. Then the risk functional is given as
\begin{align*}
Q[g] =& \mathbb{E}[(Y - g(X))^2] + \lambda \mathbb{E}(||\nabla g(X)||_2^2)
\\
=& \int (y - g(x))^2 \, d\mathbb{P}^{(Y,X)}(x,y) + \lambda \int ||\nabla g(x)||_2^2 \, d\mathbb{P}^{X}(x)
\\
=& \int (y - g(x))^2 f_{Y,X}(y,x) dxdy + \lambda \int ||\nabla g(x)||_2^2 f_X(x) dx.
\end{align*}
This is a first-order risk functional. There is a solution to the minimization problem $\min_{g \in W^{2,2}} Q[g]$ and the function $g:\Omega \rightarrow \mathbb{R}$ that minimizes the risk functional on $\Omega$ is given as the solution to the following Neumann boundary problem,
\begin{align*}
\begin{cases}
- \lambda \Delta g(x) - \lambda [\nabla \log(f_X(x)) \cdot \nabla g(x)] + g(x) = \mathbb{E}[Y|X = x] & x \in \Omega,
\\
\frac{\partial g(x)}{\partial x_i}\big|_{x_i = 0} = \frac{\partial g(x)}{\partial x_i}\big|_{x_i = 1} = 0. & 
\end{cases}
\end{align*}
Again, we additionally impose $g(0) = \mathbb{E}(Y|X = 0)$, when necessary as in e.g. Remark c) below.\\
\textbf{Remarks:}
\begin{itemize}
\item[a)] The conditions from \ref{App:AssumptionsEulerLagrange} have to be checked. We recognize that
\begin{align*}
f(x,s,u) = \int (y - s)^2f_{X,Y}(x,y) dy + \lambda ||u||_2^2 f_X(x).
\end{align*}
Then: 1) $f(x,s,u) = \int (s - y)^2 f_{X,Y}(x,y)dy + \lambda ||u||_2^2$ is a Carath\'{e}odory integrand since $f$ is measurable and continuous in all arguments; 2) $f$ is continuously differentiable in $s$ and $u$ since $f$ is quadratic in both arguments; 3) $|\partial_v f(x,s,u)| = |2 \int (s - y)f_{X,Y}(x,y)dy| \leq c_1 |s|$ and $|\nabla_u f(x,s,u)| = \lambda 2 |u| \leq c_2 |u|$; 4) $f$ is convex in $(s,u)$. So the assumptions are fulfilled and the solution to the field equation is a minimum of the risk functional.
%5) $|f(x,v,A)| = | \int (v - y)^2 f_{X,Y}(x,y)dy + \lambda ||A||_2^2| \leq | \int (v - y)^2 f_{X,Y}(x,y)dy| + \lambda ||A||_2^2 \leq c_1 \int (v-y)^2 dy + \lambda ||A||_2^2 \leq c_2 v^3 + \lambda ||A||_2^2$. Also $||A||_2^2$ is convex.
\item[b)] The field equation is a second-order, inhomogeneous, linear PDE. The inhomogeneity is the conditional expecation and the coefficient function for $\nabla g$ is the score of the joint distribution of the features, $\nabla \log(f_{X}(x))$. The PDE may be used for post regularization, i.e. having an unbiased estimator for $\mathbb{E}(Y|X = \,\cdot\,)$ which suffers from high variance, it may be used as plug-in estimator in the field equations. Then the solution to the field equation, $\tilde{g}$, will be regularized.
\item[c)] For $\lambda = 0$, we obtain $g(x) = \mathbb{E}(Y|X = x)$, as can be expected. For $\lambda \rightarrow \infty$, the equation simplifies to 
\begin{align*}
\Delta g(x) = - \nabla \log(f_X(x)) \cdot \nabla g(x).
\end{align*}
This equation is trivially solved by $g = C, \ C \in \mathbb{R}$, whereby the Neumann boundary condition is also automatically fulfilled. To determine the constant $C$, we use the auxiliary condition that $g(0) = \mathbb{E}(Y|X  = 0)$. Then $g(x) =  \mathbb{E}(Y|X  = 0), \ \forall x \in \Omega$.\\
For $f_X$ being the independent uniform distribution on $\Omega$, $f_X(x) = \frac{1}{\lambda(\Omega)} = 1$, we obtain
\begin{align*}
\mathbb{E}(Y|X = x) = - \lambda \Delta g(x) + g(x).
\end{align*}
This reflects the fact that if the distribution of $X$ is unstructured, it does not contribute to the PDE and consequently does not influence the minimum of the risk. Under this assumption and the $\ell^2$-penalty (instead of the squared-$\ell^2$-penalty) this equation is the equivalent to the Rudin-Osher-Fatemi differential equation, cf. \cite{rudin1992nonlinear}. The equation for the squared-$\ell^2$-penalty and $f_X$ being the uniform distribution is also given in \cite{ding2024semi}.
\item[d)] The one-dimensional version of the equation is given as
\begin{align}
\begin{cases}
\mathbb{E}(Y|X = x) = -\lambda g''(x) - \lambda (\log(f_X(x)))' g'(x) + g(x) & x \in \Omega
\\
g'(0) = g'(1) = 0 & 
\end{cases}
\label{OneDimFieldEq}
\end{align}
This equation is closely related to the Hermite differential equation and may be solved by a power-series ansatz.
%, cf. \cite{weisstein2016hermite}.
%TODO Hermite reference
For the special case where additionally $f_X(x) = 1$, the resulting ODE can be solved explicitly for several forms of $\mathbb{E}(Y|X = x)$. To illustrate the mechanics of the ODE, we set $\mathbb{E}(Y|X = x) = \cos(\omega x)$. Then it is easy to show that the solution, $\tilde{g}$, is given as
\begin{align*}
\tilde{g}(x) = A \sinh\left(\frac{1}{\sqrt{\lambda}} x\right) + \frac{1}{1 + \lambda \omega^2} \cos(\omega x).
\end{align*}
with $A = \frac{\sqrt{\lambda} \omega \sin(\omega)}{2 \sinh\left(\frac{1}{\sqrt{\lambda}}\right)(1 + \lambda \omega^2)}$. Notice that $\lim_{\lambda \rightarrow 0} \tilde{g}(x) = \cos(\omega x) \ \forall x \in [0,1]$ with rate $\sqrt{\lambda}$ and $\lim_{\lambda \rightarrow \infty} \tilde{g}(x) = 0 \ \forall x \in [0,1]$ with rate $\sqrt{\lambda}$. Interestingly, for $\lambda \rightarrow \infty$, the first term dominates.
\end{itemize}
\section{Discussion}
\label{Discussion}
%
%Overview
\subsection{Overview}
\label{Dicussion:Overview}
\noindent
Penalized risks in non-parametric regression settings are analyzed by means of variational calculus. When the risk is comprised of a fitness-term and a gradient-based penalty, under suitable regularity conditions, the minimum of the risk can be found as the solution to the Euler-Lagrange equations associated to the risk. The fitness-term analyzed in this paper is the mean-squared error (MSE) risk, $L[g] = \mathbb{E}((Y - g(X)^2)$. $L$ is a quadratic functional and its unique minimum is $\tilde{g}(x) = \mathbb{E}(Y|X = x)$, cf. Sect. 1.5.5 from \cite{bishop2006pattern} or the motivation sections of this paper, Sect.~\ref{Introduction:Motivation}. Adding a penalty term to the MSE-risk complicates finding the minimum of the associated risk functional. The penalty terms considered in this paper are gradient-based penalties. Especially interesting is the case where the penalty is
of the form $\text{Pen}[g] = \int ||\nabla g(x)||_{r}^s d\nu(x)$. When $s=r=2$ we may call the penalty squared-$\ell^2$-penalty. From a statistical point of view such penalty decreases the variance of an estimator $\hat{g}$ when it is used in its training. For a linear $g$, e.g. $g = \beta_0 + \sum_{i = 1}^p \beta_i x_i$, the squared-$\ell^2$ penalty is $\sum_{i = 1}^p \beta_i^2$, which corresponds to the classical ridge penalty, cf. Sect. 3.4.1 from \cite{hastie2017elements} or Sect.~\ref{Basics:PenalizedRegression:Linear} of this paper. 
From a calculus viewpoint, the squared-$\ell^2$-penalty is a Dirichlet functional if $\nu$ is the Lebesgue-measure, cf. example 2.8 from \cite{rindler2018calculus}. It is well known that the minimum of the Dirichlet functional is the solution to the Laplace equation. This fact alone motivates the realization that the minimum of a risk comprised of the MSE and a gradient-based penalty is subject to a second order partial differential equation. In fact, the differential equation for the squared-$\ell^2$-penalized MSE is derived as a linear, inhomogeneous second order PDE in Sect.~\ref{FieldEqs:squared-ell2}. In addition to the derivation for the squared-$\ell^2$ penalized MSE-risk, we derive the field equations for first-order risk functionals, as a direct consequence of the Euler-Lagrange equations. These field equations may be used for post-processing of estimators. Given a non-parametric estimator for $\mathbb{E}(Y|X =\,\cdot\,)$ with a high variance, the variance can be reduced by using the estimator as a plug-in estimator for the conditional expectation in the field equation and (numerically) solving the field equation for $\tilde{g}$ and then taking $\tilde{g}$ as the updated estimator.
%The calculation of the field equation for the $\ell^1$-penalized MSE is a bit more involved, since $x \mapsto |x|$ is not differentiable in $x = 0$. Further research needs to be conducted in this direction.
%We argue that we can calculate the derivative as $sgn(x)$ still due to the expressions being equal $\lambda$ almost everywhere. The result is a non-linear, inhomogeneous second-order PDE, whose dynamics are rather mysterious, cf. Sect.~\ref{FieldEqs:ell1}. Especially troubling is the involvement of the Dirac-$\delta$, which heavily restricts a possible solution to the field equation, to either be monotonic or to have a special restriction on its extrema, see remark (x) in Sect.~\ref{FieldEqs:ell1}.
%The solution to such a differential equation might be degenerate. This implies that a statistical learner might not learn a unique field configuration.
%
\subsection{Analysis of other penalization and regularization methods}
\label{Discussion:OtherPenalties}
\noindent
While we tried to render the results in this paper as general as possible, there are penalization/regularization methods where it is not clear how they may be analyzed using this framework.
% No penalization regularization
First and foremost, regularization methods which are not representable by a penalty have to be mentioned. Examples are dropout regularization for MLPs, cf. \cite{hinton2012improving} or best subset selection, cf. \cite{hastie2017elements}, Sect.~3.3.1. For such regularization methods there is almost no hope to formulate the associated effective risk as first (or any order) functional.\\
% Penalty but not gradient-based.
Next, there are regularization methods which are based on a penalty, but the penalty can not be formulated in terms of the gradient. E.g., \cite{follain2024nonparametric} use an $\ell^2$-type loss for Hermite polynomials but they add a coefficient to each part of the Hermite polynomial derivative series. These coefficients hinder formulating Euler-Lagrange equations for this particular risk. Also the penalty from \cite{rosasco2013nonparametric} is hard to analyze because of the restriction to a RKHS. In principle, it could be ignored that all functions are from a certain RKHS, but then the minimum of the risk functional could not be included in the RKHS.\\
% Non diffable
Last, we consider gradient-based penalties, which can not be analyzed because of non-differentiability.
An obvious example is $\ell^1$-based penalization,
\begin{align}
\text{Pen}[g] = \int ||\nabla g(x)||_1 d\nu(x).
\label{Eq:ell1Penalty}
\end{align}
At the current moment it is not clear how to deal with such set-ups because standard variational calculus methods are not applicable. Nonetheless, there is some hope because at least for the special case where $\nu = \lambda$, the minimum of the functional in Eq.~\ref{Eq:ell1Penalty} can be seen to be given by monotonic fields.\footnote{By 'monotonic field' we mean a field $g$ which is subject to either $\frac{\partial g(x)}{\partial x_i} \geq 0$ or $\frac{\partial g(x)}{\partial x_i} \leq 0$ for all $x \in \Omega$ and all $i \in \{1,\hdots,p\}$. This result may be of independent interest and will be presented in forthcoming works.} However, it is not clear yet what the minimum of the functional in conjunction with a fitness term will be. \\
Next to $\ell^1$-penalization, another interesting research question would be to analyze a generalization of the adaptive LASSO from \cite{zou2006adaptive}. A non-parametric generalization of the adaptive LASSO penalty can be conceived as
\begin{align}
\text{Pen}[g] = 
\int \sum_{i = 1}^p \frac{|\partial_i g(x)|}{\partial_i \mathbb{E}(Y|X = x)} d\nu(x),
\end{align}
The interesting question here would be if the oracle properties, cf. \cite{zou2006adaptive} Sect. 3.2, translate to the non-parametric case and if they can be uncovered in the field equations.
\subsection{Convergence to the minimum of the risk functional and universal approximation}
\label{Discussion:Overview}
\noindent
Notice that throughout this paper, we implicitly assumed that a statistical learner, when trained w.r.t an empirical risk, learns the minimum of the risk. Currently this statement is merely a conjecture. To clarify this, we first would like to make a distinction between learners for which a universal approximation theorem holds and learners for which there is no universal approximation theorem. An example for a universal learner is a feedforward MLP, cf. \cite{hornik1991approximation}. If the solution to the Euler-Lagrange equation lies in the function space which the learner can universally approximate, then we presume that the learner converges to the minimum of the Euler-Lagrange equation for the MSE with gradient-based penalization. However, for now, this is just a presumption, since the universal approximatiom theorem is a deterministic statement. It does not follow directly that the estimation is consistent.\footnote{For some risks and learners there are consistency results available in the literature.\cite{schmidt2020nonparametric} show that ReLU-networks are consistent under certain assumptions (and additionally the author derived the convergence rates) and \cite{kohler2021rate} derive the convergence rates for fully connected MLPs.}
Now if the function space of the learner is restricted and hence the learner is not subject to a universal approximation theorem on a general function space in which the solution of the Euler-Lagrange equation lies in, then the learner will not converge to the global minimum of the risk. Instead, we presume that the learner converges to the minimum of the risk functional over the restricted function space. The two presumptions are not trivial and have to be proven. This is subject to future research. A possible path would be to first establish the uniform convergence of the risks as functions of parameters and then generalize to uniform convergence of functionals.

\section*{Acknowledgement}
\noindent
This research was (partially) funded in the course of TRR 391 \textit{Spatio-temporal Statistics for the Transition of Energy and Transport} (520388526) by the Deutsche Forschungsgemeinschaft (DFG, German Research Foundation).

\clearpage
\bibliography{bibi}

@article{hoerl1970ridge,
  title={{Ridge Regression: Applications to Nonorthogonal Problems}},
  author={Hoerl, Arthur E and Kennard, Robert W},
  journal={Technometrics},
  volume={12},
  number={1},
  pages={69--82},
  year={1970},
  publisher={Taylor \& Francis}
}

@article{tibshirani1996regression,
  title={Regression shrinkage and selection via the lasso},
  author={Tibshirani, Robert},
  journal={Journal of the Royal Statistical Society Series B: Statistical Methodology},
  volume={58},
  number={1},
  pages={267--288},
  year={1996},
  publisher={Oxford University Press}
}

@article{zou2006adaptive,
  title={{The Adaptive Lasso and Its Oracle Properties}},
  author={Zou, Hui},
  journal={Journal of the American Statistical Association},
  volume={101},
  number={476},
  pages={1418--1429},
  year={2006},
  publisher={Taylor \& Francis}
}

@article{yuan2006model,
  title={Model Selection and Estimation in Regression with Grouped Variables},
  author={Yuan, Ming and Lin, Yi},
  journal={Journal of the Royal Statistical Society Series B: Statistical Methodology},
  volume={68},
  number={1},
  pages={49--67},
  year={2006},
  publisher={Oxford University Press}
}

@book{hayashi2011econometrics,
  title={Econometrics},
  author={Hayashi, Fumio},
  year={2000},
  publisher={Princeton University Press}
}

@book{hastie2017elements,
  title={{The Elements of Statistical Learning: Data Mining, Inference, and Prediction}},
  author={Hastie, Trevor and Tibshirani, Robert and Friedman, Jerome},
  year={2009},
  publisher={Springer New York, NY}
}

@book{rindler2018calculus,
  title={Calculus of Variations},
  author={Rindler, Filip},
  year={2018},
  publisher={Springer Cham}
}

@book{giaquinta1996calculus,
  title={Calculus of Variations I},
  author={Giaquinta, Mariano and Hildebrandt, Stefan},
  publisher={Springer Berlin, Heidelberg},
  year={1996}
}

@book{bishop2006pattern,
  title={Pattern {R}ecognition and {M}achine {L}earning},
  author={Bishop, Christopher M},
  publisher={Springer New York, NY},
  year={2006}
}

@article{rosasco2013nonparametric,
  title={Nonparametric sparsity and regularization},
  author={Rosasco, Lorenzo Andrea and Villa, Silvia and Mosci, Sofia and Santoro, Matteo and Verri, Alessandro},
  year={2013},
  journal={Journal of Machine Learning Research},
  volume={14},
  number={1},
  pages={1665--1714},
  publisher={Association for Computing Machinery (ACM)}
}

@book{wahba1990spline,
  title={Spline models for observational data},
  author={Wahba, Grace},
  year={1990},
  publisher={SIAM}
}

@article{drucker1992improving,
  title={Improving generalization performance using double backpropagation},
  author={Drucker, Harris and Le Cun, Yann},
  journal={IEEE transactions on Neural Networks},
  volume={3},
  number={6},
  pages={991--997},
  year={1992}
}

@inproceedings{drucker1991improving,
  title={Improving generalization performance in character recognition},
  author={Drucker, Harris and Le Cun, Yann},
  booktitle={Neural Networks for Signal Processing: Proceedings of the 1991 IEEE Workshop},
  year={1991}
}

@article{ding2024semi,
  title={Semi-{S}upervised {D}eep {S}obolev {R}egression: {E}stimation, {V}ariable {S}election and {B}eyond},
  author={Ding, Zhao and Duan, Chenguang and Jiao, Yuling and Yang, Jerry Zhijian},
  journal={arXiv preprint arXiv:2401.04535},
  year={2024}
}

@article{groll2019lasso,
  title={{LASSO}-type penalization in the framework of generalized additive models for location, scale and shape},
  author={Groll, Andreas and Hambuckers, Julien and Kneib, Thomas and Umlauf, Nikolaus},
  journal={Computational Statistics \& Data Analysis},
  volume={140},
  pages={59--73},
  year={2019},
  publisher={Elsevier}
}

@article{shen2022consistency,
  title={Consistency of neural networks with regularization},
  author={Shen, Xiaoxi and Lin, Jinghang},
  journal={arXiv preprint arXiv:2207.01538},
  year={2022}
}

@article{follain2024nonparametric,
  title={Nonparametric linear feature learning in regression through regularisation},
  author={Follain, Bertille and Bach, Francis},
  journal={Electronic Journal of Statistics},
  volume={18},
  number={2},
  pages={4075--4118},
  year={2024},
  publisher={The Institute of Mathematical Statistics and the Bernoulli Society}
}

@article{hinton2012improving,
  title={Improving neural networks by preventing co-adaptation of feature detectors},
  author={Hinton, Geoffrey E and Srivastava, Nitish and Krizhevsky, Alex and Sutskever, Ilya and Salakhutdinov, Ruslan R},
  journal={arXiv preprint arXiv:1207.0580},
  year={2012}
}

@article{stone1982optimal,
  title={Optimal Global Rates of Convergence for Nonparametric Regression},
  author={Stone, Charles J},
  journal={The Annals of Statistics},
  volume={10},
  number={4},
  pages={1040--1053},
  year={1982},
  publisher={JSTOR}
}

@article{hornik1991approximation,
  title={Approximation capabilities of multilayer feedforward networks},
  author={Hornik, Kurt},
  journal={Neural Networks},
  volume={4},
  number={2},
  pages={251--257},
  year={1991},
  publisher={Elsevier}
}

@article{schmidt2020nonparametric,
  title={Nonparametric regression using deep neural networks with {ReLU} activation function},
  author={Schmidt-Hieber, Johannes},
  journal={The Annals of Statistics},
  volume={48},
  number={4},
  pages={1875--1897},
  year={2020}
}

@article{kohler2021rate,
  title={On the rate of convergence of fully connected deep neural network regression estimates},
  author={Kohler, Michael and Langer, Sophie},
  journal={The Annals of Statistics},
  volume={49},
  number={4},
  pages={2231--2249},
  year={2021},
  publisher={Institute of Mathematical Statistics}
}

@article{rudin1992nonlinear,
  title={Nonlinear total variation based noise removal algorithms},
  author={Rudin, Leonid I and Osher, Stanley and Fatemi, Emad},
  journal={Physica D: Nonlinear Phenomena},
  volume={60},
  number={1-4},
  pages={259--268},
  year={1992},
  publisher={Elsevier}
}

@article{torres2014total,
  title={Total-variation-based methods for gravitational wave denoising},
  author={Torres, Alejandro and Marquina, Antonio and Font, Jos{\'e} A and Ib{\'a}{\~n}ez, Jos{\'e} M},
  journal={Physical Review D},
  volume={90},
  issue = {8},
  pages = {084029},
  numpages = {15}, 
  year={2014},
  publisher={APS}
}
\clearpage
\appendix
\section{Assumptions for the Euler-Lagrange Equations}
\label{App:AssumptionsEulerLagrange}
\begin{definition}[Assumptions]
\label{Def:GoodFirstOrderRiskFunctional}
For a functional $Q: W^{k,q} \rightarrow \mathbb{R}$, with
\begin{align}
Q[g] = \int L[y; g(x)]d\mathbb{P}^{(X,Y)}  + \lambda \int \text{Pen}[\nabla g] d\nu, 
\end{align}
which we may rewrite as
\begin{align}
Q[g] =& \int f(x, g(x), \nabla g(x)) dx, \ \text{where}
\\
f(x, s, u) =& \int L(y,s) f_{X,Y}(x,y) dy + \lambda \text{Pen}(u) f_{\nu}(x), 
\end{align}
We define the following 
assumptions:
\begin{itemize}
\item[A1)] $f$ is a Carath\'eodory integrand: The mapping $x \mapsto f(x,s,u)$ is Lebesgue measurable $\forall s \in \mathbb{R}, u \in \mathbb{R}^p$ and the mapping $(s,u) \mapsto f(x,s,u)$ is continuous $\lambda$ almost everywhere.
\item[A2)] $f$ is continuously differentiable in its second and third argument. 
\item[A3)] $f$ fulfills the following growth bounds: $\forall x \in \Omega, s \in \mathbb{R}, u \in \mathbb{R}^p$ there is $M > 0$ and $q \geq 1$ s.t. $|\partial_s f(x, s, u)| \leq M (1 + |s|^{q-1} + |u|^{q-1})$ and $|\nabla_u f(x, s, u)| \leq M (1 + |s|^{q-1} + |u|^{q-1})$
\item[A4)] $f(x,s,\cdot)$ is quasiconvex $\forall (x,s) \in \mathbb{R}, \mathbb{R}$
%\item[A5)] $p$-coercivity: $\mu |A|^p \leq f(x,v,A)$
\end{itemize}
We note that the assumptions are directly taken from Theorem 3.1 and proposition 3.3 from \cite{rindler2018calculus}.
\end{definition}

\end{document}